\documentclass[10pt]{amsart}
\usepackage{geometry} 
\newtheorem{theorem}{Theorem}[section]

\newtheorem{definition}{Definition}[section]
 
\newtheorem{proposition}[theorem]{Proposition}
\makeatletter
\numberwithin{equation}{section}
\makeatother

\def\R{\mathbf{R}}
\def\T{\mathbf{T}}

\def\A{\mathcal A}
\def\C{\mathcal C}
\def\C{\mathcal C}

\def\im{\hbox{im}}

\begin{document}
\title{A different look at controllability} 
\author{Pablo Pedregal}
\address{ETSI Industriales, Universidad de Castilla-La Mancha, 13071 Ciudad Real, Spain}
\email{pablo.pedregal@uclm.es}
\subjclass[2000]{}
\thanks{
E.T.S. Ingenieros Industriales. Universidad de Castilla La Mancha.
Campus de Ciudad Real (Spain). Research supported by
MTM2010-19739 of the MCyT (Spain).
 e-mail:{\tt pablo.pedregal@uclm.es
}}

\maketitle
\begin{abstract}
We explore further controllability problems through a standard least square approach as in \cite{pedregal}. By setting up a suitable error functional $E$, and putting $m(\ge0)$ for the infimum, we interpret approximate controllability by asking $m=0$, while exact controllability corresponds, in addition, to demanding that $m$ is attained. We also provide a condition, formulated entirely in terms of the error $E$, which turns out to be equivalent to the unique continuation property, and to approximate controllability.  Though we restrict attention here to the 1D, homogeneous heat equation to explain the main ideas, they can be extended in a similar way to many other scenarios some of which have already been explored numerically, due to the flexibility of the procedure for the numerical approximation. 
\end{abstract}

\section{Introduction}
We would like to explore controllability problems through a least square approximation strategy. As a matter of fact, it has already been considered and proposed in \cite{pedregal}. 

As in a typical least-square approximation, we will set up a non-negative error functional $E$, defined in a suitable space, and let $m=\inf E\ge0$. Approximate controllability is then defined by demanding $m=0$, while exact controllability occurs when, in addition, $m$ is a minimum. The main benefit we have found concerning this viewpoint is that the controllability problem is translated, in an equivalent way, into minimizing such an error functional. As such, from the numerical point of view, one can then proceed to produce numerical approximation by utilizing typical descent strategies. Especially in linear situations, such an error functional is convex and quadratic (though not necessarily coercive), and this numerical procedure should work fine. This is indeed so. 
In addition, it is also true that  $m=0$ is equivalent to the interesting property $E'=0$ implies $E=0$, so that the only possible critical value for $E$ is zero (\cite{pedregal}). 

We only treat explicitly these ideas for the homogeneous heat equation in (spatial) dimension $N=1$, as  it will be pretty clear how to extend this philosophy to many other situations. Specifically, we take $\Omega=(0,1)\subset\R$, and $T>0$, the time horizon. 
\begin{quotation}
For initial and final data $u_0(x)$, and $u_T(x)$, respectively, we would like to find the right-point condition $f(t)$, so that the solution of the heat problem
$$
u_t-u_{xx}=0\hbox{ in }(0, T)\times(0, 1),\quad u(t, 0)=0, u(0, x)=u_0(x), u(t, 1)=f(t),
$$
will comply with $u(T, x)=u_T(x)$. 
\end{quotation}
This is the boundary controllability situation. There is also an inner controllability case in which we take a fixed subdomain $\omega\subset\Omega$.
\begin{quotation}
Determine the source term $f(t, x)$ supported in $\omega$ in such a way that the solution of the problem
$$
u_t-u_{xx}=f\chi_\omega\hbox{ in }(0, T)\times(0, 1),\quad u(t, 0)=0, u(0, x)=u_0(x), u(t, 1)=0,
$$
will comply with $u(T, x)=u_T(x)$. 
\end{quotation}

We refer to \cite{FR}, and to \cite{Russell} for classic results on controllability, and to \cite{FurIma}, \cite{LebRob} for a more recent analysis. \cite{carthel} also contains important ideas through duality arguments in the context of the Hilbert Uniqueness Method. See also \cite{Coron}, \cite{glowinski08}. The numerical analysis of this kind of controllability problems has attracted a lot of work. Without pretending to be exhaustive, we would mention important contributions covering a whole range of methods and approaches in \cite{belgacem}, \cite{carthel}, \cite{EFC-AM-I}, \cite{kindermann}, \cite{micu-zuazua}, \cite{AM-EZ}.

\section{Approximate controllability}\label{prime}
To keep the formalism to a minimum without compromising rigor, let us stick to the situation described in the Introduction
by taking
$Q_T=(0, T)\times(0, 1)$, $\overline u\in H^1(Q_T)$ furnishing the data for $t=0$ and $t=T$, $\overline u(0, x)=u_0(x)$ and $\overline u(T, x)=u_T(x)$, respectively, and assuming that $\overline u(t, 0)=0$ in the sense of traces. Let 
\begin{gather}
\A_0=\{U\in H^1(Q_T): U(t, 0)=U(0, x)=U(T, x)=0 \hbox{ for }x\in(0, 1), t\in(0, T)\},\nonumber\\
 \A=\overline u+\A_0.\nonumber
\end{gather}
For $u\in\A$, define its ``corrector" $v\in H^1_0(Q_T)$ 
to be the unique solution of the variational equality
\begin{equation}\label{calor}
\int_{Q_T}[(u_t+v)\phi+(u_x+v_x)\phi_x+v_t\phi_t]\,dx\,dt=0
\end{equation}
for all $\phi\in H^1_0(Q_T)$. 
Note that $v$ is the unique solution of the minimization problem
$$
\hbox{Minimize in }w\in H^1_0(Q_T):\quad \int_{Q_T}\left(\frac12[(u_x+w_x)^2+w_t^2+w^2]+u_tw\right)\,dx\,dt,
$$
whose equilibrium equation is
\begin{equation}\label{eq}
-(u_x+v_x)_x-v_{tt}+v+u_t=0\hbox{ in }Q_T.
\end{equation}
The weak formulation of (\ref{eq}) is precisely (\ref{calor}). 
The error functional $E_T:\A\to\R^+$ is taken to be the size of the corrector
$$
E_T(u)=\int_{Q_T}\frac12 (v_x^2+v_t^2+v^2)\,dx\,dt.
$$
It is obvious to realize that if $E(u)=0$ because $u$ has a vanishing corrector, then $u$ is a solution of the heat equation (\ref{eq}) with $v\equiv0$, complying with boundary, initial, and final conditions provided by $\overline u$, i.e. the trace of $u$ over $x=1$ is the boundary control sought. Let $\T:\A\mapsto H^1_0(Q_T)$ 
be the linear, continuous operator taking $u$ into its corrector $v$. Our setting is definitely a least-square approach in the spirit of \cite{bochevgunz}, \cite{glowinski83}. 

\begin{definition}\label{defprin}
\begin{enumerate}
\item We say that the datum $\overline u(T, x)$ is approximately controllable from $\overline u(0, x)$ through the subset $\{1\}$ of the boundary of $(0, 1)$, if for every $\epsilon>0$, there is $u_\epsilon\in\A$ such that $E(u_\epsilon)<\epsilon$. 
\item We say that the datum $\overline u(T, x)$ is exactly controllable from $\overline u(0, x)$ through the subset $\{1\}$ of the boundary of $(0, 1)$, if there is $u\in\A$ such that $E(u)=0$. 
\item The unique continuation property holds, given our framework, if the only function $v\in\im\T\subset H^1_0(Q_T)$ 
with
\begin{equation}\label{unique}
\int_{Q_T}(U_tv+U_xv_x)\,dx\,dt=0
\end{equation}
for all $U\in\A_0$ is the trivial one $v\equiv0$. 
\end{enumerate}
\end{definition}
Within the framework just introduced through this definition, we can now state our main result. 
\begin{theorem}
Let $T>0$, and $\overline u\in H^1(Q_T)$ be given. The following assertions are equivalent:
\begin{enumerate}
\item the trace $\overline u(T, x)$ for $t=T$ of $\overline u$ is approximately controllable from $\overline u(0, x)$ through the end-point $\{1\}$;
\item the unique continuation property holds;
\item $E'_T(u)=0$ implies $E_T(u)=0$ for $u\in\A$.
\end{enumerate}
\end{theorem}
\begin{proof}
We will show $(2)\Longrightarrow (1)\Longrightarrow (3)\Longrightarrow(2)$.

Let $u\in\A$, and take $U\in\A_0$. Consider the variation $u+\eta U$, preserving boundary, initial, and final data. Let $v$ be the corrector associated with $u$, and put $v+\eta V$ for the corrector associated with $u+\eta U$. By linearity, it is elementary to argue that 
\begin{equation}\label{basico}
\int_{Q_T}[(U_t+V)\phi+(U_x+V_x)\phi_x+V_t\phi_t]\,dx\,dt=0
\end{equation}
for all $\phi\in H^1_0(Q_T)$, 
where $V$, as $v$ itself, belongs to $H^1_0(Q_T)$. 
On the other hand, it is also elementary to compute the derivative of $E_T(u+\eta U)$ with respect to $\eta$ at $\eta=0$. It is given by
$$
\int_{Q_T}(vV+v_xV_x+v_tV_t)\,dx\,dt.
$$
By using (\ref{basico}) for $\phi=v$, we also can write
$$
\langle E'_T(u), U\rangle=\left.\frac{d E_T(u+\eta U)}{d\eta}\right|_{\eta=0}=-\int_{Q_T}(U_tv+U_xv_x)\,dx\,dt.
$$
Recall that $\T:\A=\overline u+\A_0\mapsto H^1_0(Q_T)$ 
can be regarded as a linear, continuous operator taking $u$ into its corrector $v$. Then, because of the unique continuation property, $\langle E'(u), U\rangle=0$ for all $U\in\A_0$ if and only if $\T u=0$. Therefore, it is a standard fact in Functional Analysis that over the quotient space $\A_0/\hbox{ker}\T$ both quantities
$\|\T u\|$ and 
$$
\sup_{U\in\A_0, \|U\|=1}\langle E'_T(u), U\rangle=\sup_{U\in\A_0, \|U\|=1}\int_{Q_T}(U_tv+U_xv_x)\,dx\,dt
$$
should be equivalent norms. Hence, for some positive constant $C>0$, 
$$
\|E'_T(u)\|\equiv\sup_{U\in\A_0, \|U\|=1}\langle E'_T(u), U\rangle\ge C\|\T u\|=CE_T(u)^{1/2}.
$$
If, starting out at arbitrary $u^0\in\A$, we follow the flow of $-E'_T$, we would eventually reach a certain $u\in\A$ so that $\|E'_T(u^0+u)\|<\epsilon$. This, together with the previous inequality, yields the approximate controllability result. 

Assume now that $\tilde u\in\A$ in a critical point of $E_T$. Under the approximate controlability property, we would like to conclude that $u$ is indeed a solution of the controlability problem. To this end, notice that:
\begin{itemize}
\item the infimum of $E_T$ over $\A$ vanishes: this is the approximate controlability property;
\item $E_T$ is a non-negative, convex functional.
\end{itemize}
As a consequence of the convexity, if $\tilde u$ is a critical point of $E_T$, it has to be a minimizer as well. But then the infimum becomes a minimum, and it has to vanish, i.e. $E_T(\tilde u)=0$.

Finally, let $\tilde v\in \im\T$ be such that (\ref{unique}) holds for all $U\in\A_0$. Let $\tilde u\in\A$ be such that $\T\tilde u=\tilde v$. By the computations performed above, (\ref{unique}) implies $E'_T(\tilde u)=0$, and so, by hypothesis, $\|\tilde v\|^2=E_T(\tilde u)=0$, that is $\tilde v\equiv0$. 
\end{proof}

The equivalence stated in this theorem implies that the above concept of approximate controllability might be a bit more flexible than the classic one, at least for data sets which are traces for $t=0$ and $t=T$ of $H^1(Q_T)$-functions.

\begin{proposition}\label{ucp}
For every positive time $T>0$, the unique continuation property in Definition \ref{defprin} holds. 
\end{proposition}
\begin{proof}
Simply notice that $H^1_0(Q_T)\subset\A_0$, and so we can take $U=v$ in (\ref{unique}) to obtain
$$
\int_{Q_T}(v_tv+v_x^2)\,dx\,dt=0.
$$
Because $v(0, x)=v(T, x)=0$ for all $x\in(0, 1)$, we conclude that $v_x\equiv0$ in $Q_T$. This together with the vanishing boundary conditions $v(t, 0)=v(t, 1)=0$ implies $v\equiv0$. 
\end{proof}

\section{Extension}
The setting described in the previous section admits some straightforward variations.
The choice of the space for the correctors $v\in H^1_0(Q_T)$ can be changed. For example, one can take $v\in H^1(Q_T)$ for a broader situation, and in this case the corrector will enjoy the natural boundary condition all around $Q_T$: $u_x=0$ for $x=0$ and $x=1$, while $u_t=0$ for $t=0$ and $t=T$. But intermediate alternatives are also possible: $v=0$ for $t=0$, and $t=T$, and so $u_x=0$ for $x=0$ and $x=1$, or $v\in\A_0$, as well. Another possibility is to define the corrector $v$ for a.e. time slice as a minimization problem only in space. This can again be easily set up in more or less the same terms (see \cite{pedregal}). 

Rather than considering these various possibilities which are straightforward variations, 
we would like to explore the most general framework that this approach may allow for $u$ instead of for $v$. Our ambient space will now be 
\begin{align}
\A_0=\{ &U\in L^2(0, T; \tilde H^1_0(0, 1)): U_t\in L^2(0, T; \tilde H^{-1}(0, 1)), \nonumber\\
&U(0, x)=U(T, x)=0 \hbox{ for }x\in(0, 1), t\in(0, T)\}.\nonumber
\end{align}
We are taking here
$$
\tilde H^1_0(0, 1))=\{U\in H^1(0, 1): U(0)=0\},
$$
while $\tilde H^{-1}(0, 1)$ is its dual. 
Notice that every $U\in\A_0$ belongs to the space $\C([0, T]; L^2(0, 1))$ so that traces of $U$ are defined for every time $t\in[0, T]$ (\cite{EvansB}). 
If $\overline u\in L^2(0, T; \tilde H^1_0(0, 1))$  (and so $\overline u(t, 0)=0$ for a.e. $t\in(0, T)$), with $\overline u_t\in L^2(0, T; \tilde H^{-1}(0, 1))$, furnishes  initial and final data, we will put as before $\A=\overline u+\A_0$. This time initial and final data, $\overline u(0, x)$, $\overline u(T, x)$ merely belong to $L^2(0, 1)$. 

For $u\in\A$, define its ``corrector" $v\in H^1_0(Q_T)$ 
to be the unique solution of the variational problem
$$
\int_{Q_T}[(u_x+v_x)\phi_x+(v_t-u)\phi_t+v\phi]\,dx\,dt=0
$$
for all $\phi\in H^1_0(Q_T)$. 
Note that $v$ is the unique solution of the minimization problem
$$
\hbox{Minimize in }w\in H^1_0(Q_T):\quad \int_{Q_T}\left(\frac12[(u_x+w_x)^2+(w_t-u)^2+w^2]\right)\,dx\,dt.
$$
The error functional $E:\A\to\R^+$ is taken to be the size of the corrector
$$
E(u)=\int_{Q_T}\frac12 (v_x^2+v_t^2+v^2)\,dx\,dt.
$$
As above, we investigate the derivative of the error functional. To this end, put $u+\eta U$ for $U\in\A_0$, and $v+\eta V$, its corresponding corrector, with $v$ the corrector for $u$. Then
$$
\int_{Q_T}[(U_x+V_x)\phi_x+(V_t-U)\phi_t+V\phi]\,dx\,dt=0
$$
for all $\phi\in H^1_0(Q_T)$. In the same way,
$$
\langle E'(u), U\rangle=\int_{Q_T}(v_xV_x+v_tV_t+vV)\,dx\,dt=\int_{Q_T}(Uv_t-U_xv_x)\,dx\,dt,
$$
by taking $\phi=v$ in the last identity. 

The unique continuation property, and the equivalence with approximate controllability are established in the same way as before. 

\section{Exact controllability}
Within this framework, exact controllability can be deduced as a consequence of the fact that the range of the map $\T$ is closed, in addition to the unique continuation property. More precisely, recall the definition of the map $\T:\overline u+\A_0\mapsto H^1_0(Q_T)$, taking every feasible $u\in\overline u+\A_0$ into its corrector $v$, in the analytical framework of Section \ref{prime}. The error functional corresponds to the least-square problem
$$
\hbox{Minimize in }u\in\overline u+\A_0:\quad \frac12\|\T u\|^2.
$$
Exact controllability can then be achieved as a consequence of two facts:
\begin{enumerate}
\item the infimum $m$ is in fact a minimum;
\item $m$ does vanish. 
\end{enumerate}
The unique continuation property is related to the second issue, but the first is equivalent to the fact that the range of $\T$ is closed. Except for general remarks involving the adjoint operator $\T^*$, the exact controllability issue involves subtle and delicate ideas about Carleman inequalities and observability. This elegant theory is very well established (see some of the references indicated in the Introduction). 

\section{The inner controllability case}
Let $\omega\subset(0, 1)$ be an interval. Put $Q_T=(0, T)\times(0, 1)$, $q_T=(0, T)\times\omega$. Let $\A_0$ be taken now as the space
$$
\A_0=\{U\in L^2(0, T; H^2(0, 1)\cap H^1_0(0, 1)): U_t\in L^2(Q_T), U=0\hbox{ on }\partial Q_T\}, 
$$
and $\overline u\in L^2(0, T; H^2(0, 1)\cap H^1_0(0, 1))$ with $\overline u_t\in L^2(Q_T)$, carrying the boundary (around $\partial\Omega$), initial, and final data. For feasible functions $u\in\A\equiv\overline u+\A_0$, let $v$ be its associated corrector,  the unique solution of the problem
\begin{equation}\label{lap}
v_{tt}+v_{xx}=\chi_{Q_T\setminus q_T}(t, x)\left(u_t-u_{xx}\right)\hbox{ in }Q_T,\quad v=0\hbox{ on }\partial Q_T,
\end{equation}
and the error functional $E_T:\A\mapsto\R^+$
$$
E_T(u)=\frac12\int_{Q_T}(v_t^2+v_x^2)\,dx\,dt.
$$
We also put $\T:\A\mapsto H^1_0(Q_T)$ for the linear, continuous mapping taking $u$ into its corrector $v$. 

Assume that $u\in\A$ is such that $E(u)=0$. Then $v\equiv0$, and 
$$
\chi_{Q_T\setminus q_T}\left(u_t-u_{xx}\right)\equiv0\hbox{ in }Q_T.
$$
Hence if we take $f=(u_t-u_{xx})\in L^2(Q_T)$, then
$$
u_t-u_{xx}=\chi_{q_T}f\hbox{ in }Q_T,
$$
and the restriction $f$ becomes the sought control. 
\begin{definition}\label{defprinseg}
\begin{enumerate}
\item We say that the datum $\overline u(T, x)$ is approximately controllable from $\overline u(0, x)$ through the subset $\omega$ of $\Omega$, if for every $\epsilon>0$, there is $u_\epsilon\in\A$ such that $E(u_\epsilon)<\epsilon$. 
\item We say that the datum $\overline u(T, x)$ is exactly controllable from $\overline u(0, x)$ through the subset $\omega$ of $\Omega$, if there is $u\in\A$ such that $E(u)=0$. 
\item The unique continuation holds if the only $v\in\im\T\subset H^1_0(Q_T)$  
with $v\equiv0$ in $q_T$ and 
\begin{equation}\label{uniquein}
\int_{Q_T\setminus q_T}(-Uv_t+U_xv_x)\,dx\,dt=0
\end{equation}
for all $U\in\A_0$ is the trivial one $v\equiv0$. 
\item We say that $E_T$ is an error functional if $E'_T(u)=0$ implies $E_T(u)=0$.
\end{enumerate}
\end{definition}

Just as in the boundary situation, the integral occurring in the unique continuation property is precisely the integral that appears when computing the Gateaux derivative
$$
\left.\frac{d E_T(u+\epsilon U)}{d\epsilon}\right|_{\epsilon=0}=\langle E'_T(u), U\rangle.
$$
Indeed, because of linearity, 
\begin{equation}\label{linealint}
V_{tt}+V_{xx}=\chi_{Q_T\setminus q_T}(U_t-U_{xx})\hbox{ in }Q_T, \quad V=0\hbox{ on }\partial Q_T,
\end{equation}
if $V$ is the variation produced in $v$ by $U\in\A_0$ in $u$. 
Then
$$
\langle E'_T(u), U\rangle=\int_{Q_T}(v_tV_t+v_xV_x)\,dx\,dt.
$$
By using (\ref{linealint}), we obtain
$$
\langle E'_T(u), U\rangle=-\int_{Q_T\setminus q_T}(U_t-U_{xx})v\,dx\,dt.
$$
Let us focus on the second term
$$
\int_{Q_T\setminus q_T}U_{xx}v\,dx\,dt.
$$
A first integration by parts yields
$$
-\int_{Q_T\setminus q_T}U_xv_x\,dx\,dt+\int_{\partial(Q_T\setminus q_T)}U_xv\,dS.
$$
But since $v=0$ around $\partial Q_T$, we find that the boundary integral equals
$$
\int_{\partial q_T} U_xv\,dS=-\int_{q_T}U_{xx}v\,dx\,dt.
$$
We can always take $U$ to be arbitrary in $q_T$, and independent of time, so that $v\equiv0$ in $q_T$. Hence, altogether,
\begin{equation}\label{derivada}
\langle E'_T(u), U\rangle=-\int_{Q_T\setminus q_T}(U_tv+U_xv_x)\,dx\,dt.
\end{equation}
This is the basic computation for an equivalence as in the  boundary situation. The proof follows exactly along the same lines as with the boundary counterpart. 
\begin{theorem}
Let $T>0$, and $\omega\subset\Omega$ be given. Let also $\overline u\in\A$ furnish initial and final data. The following are equivalent:
\begin{enumerate}
\item the trace $\overline u(T, x)$ for $t=T$ of $\overline u$ is approximately controllable from $\overline u(0, x)$ through $\omega$;
\item the corresponding unique continuation holds;
\item $E_T$ is an error functional in the sense of Definition \ref{defprinseg}.
\end{enumerate}
\end{theorem}
In this setting, it is also immediate to check that the unique continuation condition holds, so that we have approximate controllability as well. Just notice that  the corrector $v$, being the solution in (\ref{lap}), is a feasible direction $U$ because $H^2(Q_T)\cap H^1_0(Q_T)\subset\A_0$. By taking $U=v$ in (\ref{derivada}), we conclude immediately that $v\equiv0$. 

\section{Final comments}
The formalism introduced here, and described in detail for the   linear, homogeneous heat equation in (spatial) dimension $N=1$ can be formally extended, in a rather direct way, to many other frameworks because of its flexibility. The specific treatment of the unique continuation property may however change from situation to situation. For instance, it is well-known that for the wave equation the unique continuation property requires a certain size of the horizon $T$ due to the finite speed of propagation. Some of those situations include, but are not limited to: 
\begin{itemize}
\item higher dimension $N>1$;
\item inhomogeneous heat equation;
\item wave equation;
\item systems of differential equations;
\item situations for degenerate equations;
\item non-linear problems.
\end{itemize}

Especially in linear cases, this viewpoint naturally leads to an iterative approximation scheme based on a standard descent method. It has already been tested in various scenarios and, at least numerically, it performs very well (see \cite{arandapedregal}, \cite{AM}, \cite{AM2}, \cite{AM-PP}, \cite{AM-PP2}).

\end{document}